\documentclass[]{jgcc}
\pdfoutput=1
\usepackage{lastpage}
\jgccdoi{16}{1}{5}{11330}  
\jgccheading{}{\pageref{LastPage}}{}{}{May.~16,~2023}{Jul.~24,~2024}{}

\usepackage[utf8]{inputenc}
\usepackage{mathabx}
\usepackage[english]{babel}
\usepackage{parskip}
\usepackage{tikz}
\usetikzlibrary{decorations.markings,arrows}
\usepackage[onehalfspacing]{setspace}
\usepackage{caption}
\usepackage{subcaption}
\usepackage{graphicx}
\usepackage{mathabx}
\usepackage{amssymb}
\usepackage[all]{xy}

\makeindex

\title{Groups of F-Type}

\author[B. Fine, G. Rosenberger and L. Wienke]{Benjamin Fine, Gerhard Rosenberger and Leonard Wienke}
\address{Department of Mathematics, University of Hamburg, Bundesstraße 55, 20146 Hamburg, Germany}
\email{gerhard.rosenberger@uni-hamburg.de}
\address{Department of Mathematics, University of Bremen, Bibliothekstraße 5, 28359 Bremen, Germany}
\email{wienke@uni-bremen.de}

\DeclareMathOperator{\rank}{rank}

\DeclareMathOperator{\tr}{tr}

\DeclareMathOperator{\PSL}{PSL}
\DeclareMathOperator{\vcd}{vcd}

\DeclareRobustCommand{\gobblefour}[4]{}

\begin{document}

\begin{abstract}We consider a class of groups, called groups of F-type, which includes some known and important classes like Fuchsian groups of geometric rank $\ge 3$, surface groups of genus $\ge 2$, cyclically pinched one-relator groups and torus-knot groups, and prove algebraic and geometric properties of these groups.
\end{abstract}
\keywords{group presentations, groups of F-type, large groups, virtually torsion-free, coherent, conjugacy separable, subgroup separable, SQ-universal, Tits alternative, essential representations, free products with amalgamation, Euler characteristic.}
\maketitle
\section{Introduction} A group $G$ is of \textit{of F-type}, if it admits a presentation of the following form
\[G=\langle a_1,\dots, a_n\mid a_1^{e_1}=\dots=a_n^{e_n}=U(a_1,\dots,a_p)V(a_{p+1},\dots,a_n)=1\rangle\]
where $n\ge 2, e_i=0$ or $e_i\ge 2$, for $i=1,\dots, n$, $1\le p\le n-1$, $U(a_1,\dots, a_p)$ is a cyclically reduced word in the free product on $a_1,\dots,a_p$ which is of infinite order and $V(a_{p+1},\dots,a_n)$ is a cyclically reduced word in the product on $a_{p+1},\dots,a_n$ which is of infinite order. With $p$  understood we write $U$ for $U(a_1,\dots,a_p)$ and $V$ for $V(a_{p+1},\dots, a_n)$. 

Now, if $U=a_1^{\pm 1}$ then $e_1$ must equal zero since we assume that $U$ has infinite order. In this case, $G$ reduces to
\[ G=\langle a_1,\dots, a_n\mid a_2^{e_2}=\dots=a_n^{e_n}=1\rangle \]
which is a free product of cyclic groups. Therefore if $p=1$ (or $p=n-1$) we restrict groups of F-type to those where $U=a_1^m$ (or $V=a_n^m$) with $m\ge 2$.

It follows then that in all cases the group $G$ decomposes as a non-trivial free product with amalgamation:
\[G=G_1\Asterisk_A G_2\]
with $G_1\neq A\neq G_2$ where the factors are free products of the cyclics
\[ G_1=\langle a_1,\dots, a_p\mid a_1^{e_1}=\dots=a_p^{e_p}=1\rangle,\]
\[ G_2=\langle a_{p+1},\dots, a_n\mid a_{p+1}^{e_{p+1}}=\dots=a_n^{e_n}=1\rangle,\]
and
\[A=\langle U^{-1}\rangle =\langle V\rangle.\]

We first mention the observation that a group of F-type is \textit{coherent}, that is, finitely generated subgroups are finitely presented. Observe that finitely generated subgroups of $G_1$ and $G_2$, respectively, are finitely related and every subgroup of $A$ is finitely generated, in fact cyclic because $A$ is cyclic. Hence, $G$ is coherent by \cite{128}.

In what follows we make a further restriction. Suppose $UV$ omits some generator. For instance, suppose that $UV$ does not involve $a_1$. Then $G$ is a free product $H_1\ast H_2$, where \[H_1=\langle a_1\mid a_1^{e_1}=1\rangle\] and \[H_2=\langle a_2,\dots, a_n\mid a_2^{e_2}=\dots =a_n^{e_n}=UV=1\rangle.\]
This does not affect the validity of the upcoming results in the next sections. Hence, we may assume that $UV$ involves all the generators. One could assume that it would be convenient to have $U$ and $V$ of minimal length in their respective orbits under Nielsen transformations but unfortunately in doing so, we eventually lose the property that $G_1$ and $G_2$ are free products on the sets of the new generators, especially if $e_i\ge 2$ for some $i$. It is important for our purposes that the factors $G_1$ and $G_2$ are free products on exactly the given generators, respectively.

In the following sections we consider essentially faithful representations in general and for groups of F-type and discuss many consequences of this. Then we classify the hyperbolic groups of F-type. The hyperbolic groups of F-type have a faithful representation into $\PSL(2,\mathbb{R})$ which gives interesting consequences for groups of F-type. Then we give some further algebraic properties and finally some quotients of groups of F-type.

This paper is in most parts a survey but contains also new results. The new results often appear in combination with already known results and are given together in one statement. These include for instance Theorem 3.2, Corollaries 3.3 and 3.5, and Theorems 4.17 to 4.19. Many results have appeared in different publications but it would be convenient to have the known results in just one place. The reason is that groups of F-type still form an interesting class of groups, especially in connection with covering space theory and algebraic topology. We have mentioned some of the many interesting examples in the abstract of this paper. As background for the used notations, definitions, and results one may take \cite{0}, \cite{149}, and \cite{152}. Groups of F-type were originally introduced in \cite{78} and \cite{81}. A reasonable discussion can be found in \cite{82}.

Sadly, our coauthor and friend Ben Fine died during the final preparation of the paper. The remaining two authors dedicate the paper to the memory of Ben and to all the work that he inspired.

\section{Essential Representations and Algebraic Consequences}

We say that a linear representation $\rho$ over a field of characteristic 0 is an \textit{essentially faithful representation} if $\rho$ is finite-dimensional with torsion-free kernel. In the following, we give some historic remarks on the notion of essentially faithful representations. The term was introduced in 1996 by B.~Fine (see \cite{E1}). Earlier, in 1985, he introduced the weaker concept of \textit{essential representations}. Suppose $G$ is a group with presentation \[G=\langle a_1,\dots, a_n\mid a_1^{e_1}=\dots=a_n^{e_n}=R_1^{m_1}=\dots=R_k^{m_k}=1\rangle\]
where $e_i=0$ or $e_i\ge 2$ for $i=1,\dots, n$, $m_j\ge 1$ for $j=1,\dots,k$, and each $R_j$ is a cyclically reduced word in the free product of the cyclic groups $\langle a_1\rangle, \dots, \langle a_m\rangle$ of syllable length at least two.

A representation $\rho\colon G\to \text{Linear Group}$ over a field of characteristic zero is an essential representation if for each $i$ the image $\rho(a_i)$ is of infinite order if $e_i=0$ or of order $e_i$ if $e_i\ge2$, and for each $j$ the image $\rho(R_j)$ has order $m_j$. This last term was used by B.~Fine, J.~Howie, R.~Hidalgo, N.~Kopteva, F.~Levin, G.~Rosenberger, R.~Thomas, E.\,B.~Vinberg and others in their work on generalized modular groups, generalized triangle groups, generalized tetrahedron groups, one-relator quotients of free products with amalgamation, groups of special NEC-type and Coxeter groups. In their paper \cite{E2}, G.~Baumslag, J.\,W.~Morgan and P.~Shalen described this phenomenon for generalized triangle groups as special representations.

We first give some general statements and show then that a group of F-type has an essentially faithful representation into $\PSL(2,\mathbb{C})$.

\begin{prop}Let $G$ be a finitely generated group. Then $G$ admits an essentially faithful representation if and only if $G$ is virtually torsion-free.
\end{prop}

\begin{proof}Suppose $G$ is finitely generated and $\rho\colon G\to \text{Linear Group}$ is an essentially faithful representation. Since $G$ is finitely generated, $\rho(G)$ is a finitely generated linear group. From a result of Selberg, see \cite{200}, $\rho(G)$ is then virtually torsion-free. Let $H$ be a torsion-free normal subgroup of $\rho(G)$ of finite index and let $H^\ast$ be the pullback of $H$ in $G$. $H^\ast$ has finite index in $G$.

If $g\neq 1$ has finite order then $\rho(g)$ has exactly the same order since $\rho$ has torsion-free kernel. Therefore $g$ cannot be in $H^\ast$ since its image would then be an element of finite order in the torsion-free group $H$. Thus $H^\ast$ must be torsion-free and $G$ is virtually torsion-free.

Conversely, suppose $G$ is virtually torsion-free. Let $H$ be a torsion-free subgroup of $G$ of finite index. The intersection of the conjugates of $H$ in $G$ is a normal subgroup of finite index. Hence, $G$ must contain a torsion-free normal subgroup $H^\ast$ of finite index. Choose a faithful finite-dimensional representation $\rho^\ast$ of the finite group $G/H^\ast$. The composition of this with the natural homomorphism from $G$ to $G/H^\ast$ will give the desired representation.
\end{proof}

This has two immediate relations to the following.

\begin{prop}\label{prop2.2}\begin{enumerate}
\item[1.] Let $G$ be a finitely generated group. Then $G$ is residually finite if and only if for each $g\in G$, $g\neq 1$, there exists a non-trivial linear representation $\rho$ of $G$ with $\rho(G)\neq \{1\}$.
\item[2.] Let $G$ be a finitely generated group with a balanced representation
\[ G=\langle a_1,\dots, a_n\mid R_1^{m_1}=\dots=R_n^{m_n}=1\rangle \]
where each $R_i$ is a non-trivial cyclically reduced word in the free product on $\{a_1,\dots,a_n\}$. If at least one $m_j\ge 2$, then $G$ is non-trivial.
\end{enumerate}
\end{prop}

For Proposition \ref{prop2.2}.1 there is nothing to show. It is a consequence of the Theorem of A.~Malcev, see \cite{E3}. We just mention this for completeness.

We now give a proof for Proposition \ref{prop2.2}.2.

\begin{proof}Assume that $m_n\ge 2$. Consider the group \[G^\ast=\langle a_1,\dots, a_n\mid R_1^{m_1}=\dots=R_{n-1}^{m_{n-1}}=1\rangle\] where $m_1,\dots, m_{n-1}\ge 1$. Its Abelianization has torsion-free rank at least 1. If we adjoin $R_n^{m_n}$ with $m_n\ge 2$ to the Abelianization, the resulting Abelian group is non-trivial.
\end{proof}

We now show that groups of F-type admit an essentially faithful representation into $\PSL(2,\mathbb{C})$. More concretely, we show the following.

\begin{thm}\label{theo3}Let $G$ be a group of F-type. Then $G$ has a representation $\rho\colon G\to \PSL(2,\mathbb{C})$ such that $\rho|_{G_1}$ and $\rho|_{G_2}$ are faithful.

Further, if neither $U$ nor $V$ is a proper power then $G$ has a faithful representation in $\PSL(2,\mathbb{C})$.
\end{thm}

\begin{proof}Choose faithful representations \[\rho_1\colon G_1\to \PSL(2,\mathbb{C})\text{ and }\rho_2\colon G_2\to \PSL(2,\mathbb{C})\] such that \[\rho_1(U^{-1})=\begin{pmatrix}
t_1 & 0\\
0& t_1^{-1}
\end{pmatrix}\text{ and } \rho_2(V)=\begin{pmatrix}
t_2 & 0\\
0& t_2^{-1}
\end{pmatrix}\]

where $t_1$ and $t_2$ are transcendental. This may be done since $U$ and $V$ both have infinite order.

If both $G_1$ and $G_2$ are cyclic, then we may choose $\rho(G)$ to be a cyclic group. Now, let $G_1$ or $G_2$ be non-cyclic. Then there exists an irreducible representation from $G$ in $\PSL(2,\mathbb{C})$. We have that each of the $n$ matrices have at least two degrees of freedom with the trace and determinant being specified.

Therefore, from the work of M.~Culler and P.~Shalen, see \cite{E4}, or also from \cite{E5} and \cite{E6} in a different setting, the dimension of the character space, that is, the representation space modulo conjugation, of $G$ in $\PSL(2,\mathbb{C})$ is at least $2n-1-3=2n-4$ which is positive for $n\ge 3$. Here, the $-1$ represents a possible additional conjugation of $G_1$ or $G_2$ using the fundamental theorem of algebra. This, especially, implies that we may choose $\rho_1$ and $\rho_2$ such that $t_1=t_2$. Now define $\rho\colon G\to\PSL(2,\mathbb{C})$ via $\rho_i=\rho|_{G_i}$ for $i=1,2$.

This gives the desired representation of the theorem. Further if neither $U$ nor $V$ is a proper power, the above construction leads to the existence of a faithful representation of $G$ because $\rho(g)$, $g\in G_1\setminus\langle U\rangle$, and $\rho(U)$ have no common fixed point, considered as a linear fractional transformation, analogously for $G_2$ and $V$. This gives the result that $G$ has a faithful representation in $\PSL(2,\mathbb{C})$.\end{proof}

We give the following consequence of Theorem \ref{theo3}.

\begin{cor}\label{cor2.4}Let $G$ be a group of F-type. Then $G$ admits an essentially faithful representation into $\PSL(2,\mathbb{C})$.
\end{cor}

We note that if both $U$ and $V$ are proper powers then there is no faithful representation in $\PSL(2,\mathbb{C})$. If $U=U_1^\alpha$, $\alpha\ge 2$, and $V=V_1^\beta$, $\beta\ge 2$, then $\rho(U_1)$ and $\rho(V_1)$ commute but $U_1$ and $V_1$ do not commute.

If $A$ and $B=A^k$, $|k|\ge 2$, are elements of infinite order in $\PSL(2,\mathbb{C})$, then they have the same fixed points, considered as linear fractional transformations, and hence commute.

In the following corollary, we describe some straightforward algebraic consequences which we get from the essentially faithful representation of a group of F-type into $\PSL(2,\mathbb{C})$ in Corollary \ref{cor2.4}.

\begin{cor}\label{cor2.5} Let $G$ be a group of F-type. Then
\begin{enumerate}
\item[1.] $G$ is virtually torsion-free.
\item[2.]If neither $U$ nor $V$ is a proper power then $G$ is residually finite and thus Hopfian.
\item[3.]\begin{enumerate}
\item[(i)]If $e_i\ge 2$ then $a_i$ has order exactly $e_i$.
\item[(ii)]Any element of finite order in $G$ is conjugate to a power of some $a_i$.
\item[(iii)]Any finite subgroup is cyclic and conjugate to a subgroup of some $\langle a_i\rangle$.
\item[(iv)] Any Abelian subgroup is cyclic or free Abelian of rank 2.
\end{enumerate}
\end{enumerate}
\end{cor}
We remark that Corollary \ref{cor2.5}.1 is follows directly from Corollary \ref{cor2.4}. Corollary \ref{cor2.5}.2 is a consequence of A.~Malcev's Theorem \cite{E3} and Corollary \ref{cor2.5}.3 follows straightforward from the Nielsen cancellation method in free products with amalgamation as described in \cite{0}, see also Theorem 3.2 for a more general situation.

The next result (Corollary \ref{cor2.6}) is concerned with the Tits alternative, that is, a group satisfies the Tits alternative if it either contains a free subgroup of rank 2 or is virtually solvable. We use for the Tits alternative some simple facts.

\begin{rem}\label{rem2.6}
\begin{enumerate}
\item[1.] A subgroup $H$ of $\PSL(2,\mathbb{C})$ is \textit{elementary} if the commutator of any two elements of infinite order has trace 2; or equivalently, $G$ is elementary if any two elements of infinite order, regarded as linear fractional transformations, have at least one common fixed point. A non-elementary subgroup contains a free subgroup of rank two. Hence, if $G$ is a group of F-type and $\rho\colon G\to \PSL(2,\mathbb{C})$ an essentially faithful representation such that $\rho(G)$ is non-elementary, then $G$ has a free subgroup of rank 2.
\item[2.] If $G=\langle a_1, a_2\mid a_1^{p}=a_2^q=1\rangle$ with $2\le p,q$ and $p+q\ge 5$ then $G$ contains a free subgroup of rank 2.
\end{enumerate}
\end{rem}
More details can be found in \cite{0}. We now give the following reasoning. In many cases we get representations $\rho\colon G\to \PSL(2,\mathbb{C})$ with $\rho(G)$ non-elementary. Then we have Corollary~\ref{cor2.6}. If each time $\rho(G)$ is elementary then we may apply Remark \ref{rem2.6}.2. or we may consider factor groups. For instance, let \[G=\langle a,b,c\mid b^2=c^2=a^s(bc)^t=1\rangle\] with $s>2$ or $t\ge 2$. Then $G$ has the factor group
\[\overline G=\langle a,b,c\mid a^s=b^2=c^2=(bc)^t=1\rangle\] which certainly has a free subgroup of rank 2. In an analogous manner, we may consider the remaining groups.

\begin{cor}\label{cor2.6}Let $G$ be a group of F-type. Then either $G$ has a free subgroup of rank 2 or is solvable and isomorphic to a group with one of the following representations:
\begin{enumerate}
\item[(1)]$H_1=\langle a,b \mid a^2b^2=1\rangle$,
\item[(2)]$H_2=\langle a,b,c \mid a^2=b^2=abc^2=1\rangle$, and
\item[(3)]$H_3=\langle a,b,c,d \mid a^2=b^2=c^2=d^2=abcd=1\rangle$.
\end{enumerate}
\end{cor}

From our previous result we get a close tie between the existence of non-Abelian free subgroups and SQ-universality. Recall that a group $H$ is \textit{SQ-universal} if every countable group can be embedded isomorphically as a subgroup of a quotient of $H$. SQ-universality is related to the concept of large groups (in the sense of S.~Pride).

A \textit{large group} is a group with a finite index subgroup that maps onto the free group $F_2$ of rank 2. We know that $F_2$ is SQ-universal, see for instance \cite{0}. Now in \cite{E7}, P.\,M.\,~Neumann showed the following. If $G$ is a subgroup of finite index in a group $G$, then $G$ is SQ-universal if and only if $H$ is SQ-universal. Hence, altogether, large groups are SQ-universal.

\begin{thm} Let $G$ be a group of F-type. If $G$ is not solvable then $G$ is \textit{large}. In particular, a group of F-type is either SQ-universal or solvable.
\end{thm}

\begin{proof}Suppose $G$ is not solvable. We may assume without loss of generality that each $e_i\ge 2$. Let $\rho\colon G\to \PSL(2,\mathbb{C})$ be an essentially faithful representation, and let $N$ be a normal torsion-free subgroup of finite index in $\rho(G)$. Let $\pi$ be the canonical epimorphism from $\rho(G)$ onto the finite group $\rho(G)/N$. Now consider \[X=\langle a_1,\dots, a_n\mid a_1^{e_1}=\dots=a_n^{e_n}=1\rangle.\] There is a canonical epimorphism $\varepsilon\colon X\to G$. Consider the sequence
\[\xymatrix{
  X \ar[r]^-{\varepsilon} &  \ar[r]^-{\rho} G & \rho(G) \ar[r]^-{\pi}  & \rho(G)/N.
}\]

This yields an epimorphism $\psi\colon X\to \rho(G)/N$. Let $Y=\ker(\psi)$. Then $Y$ is a normal subgroup of finite index $j$ in $X$, and $Y\cap \langle a_i\rangle=\{1\}$ for $i=1,\dots, n$. Then by the Kurosh theorem $Y$ is a free group of finite rank $r$. The finitely generated free product of cyclic groups $Y$ may be considered as a Fuchsian group of finite hyperbolic area.

From the Riemann-Hurwitz formula we have \[j\left(n-1-\left(\frac{1}{e_1}+\dots+\frac{1}{e_n}\right)\right)=r-1.\]
Therefore we have \[r=1-j\left(\left(\frac{1}{e_1}+\dots+\frac{1}{e_n}\right)-n+1\right).\]
The group $G$ is obtained from $X$ by adjoining the additional relation $UV=1$ and thus $G=X/K$ where $K$ is the normal closure of $UV$ in $X$. Since $K\subset Y$ the factor group $Y/K$ may be regarded as a subgroup of some finite index $j$ in $G$, and using the Reidemeister-Schreier method, $Y/K$ can be defined on $r$ generators and $j$ relations. Then the deficiency of this presentation is given by
\begin{align*}d&=r-j=1-j\left(\left(\frac{1}{e_1}+\dots+\frac{1}{e_n}\right)-n+2\right)\\
&=1+j\left(n-2-\left(\frac{1}{e_1}+\dots + \frac{1}{e_n}\right)\right).
\end{align*}
If $n\ge 5$ or $n=4$ and at least one $e_i\neq 2$, then $d\ge 2$. It follows then from \cite{9} that $G$ contains a subgroup of finite index which maps onto a free group of rank 2. Hence, $G$ is large and therefore SQ-universal. Next suppose $n=4$ and all $e_i=2$. Then necessarily $p=2$ and $U=(a_1a_2)^s, V=(a_3a_4)^t$ with $|s|,|t|\ge 1$. Since $G$ is non-solvable, then without loss of generality we may assume that $s\ge 2$ and $t\ge 1$. Then $G$ has as a factor group the free product 
\[\overline{G}=\langle a_1, a_2, a_3\mid a_1^2=a_2^2=(a_1a_2)^s=a_3^2=1\rangle .\]
The group $\overline{G}$ has as a normal subgroup of index 2 a group that is isomorphic to \[\langle x,y,z\mid x^s=y^2=z^2=1\rangle.\] Therefore, $\overline{G}$ and hence $G$ also has a subgroup of finite index mapping onto a free group of rank 2. If $n=3$ then the result follows by a similar case-by-case consideration, see \cite[Chapter 8]{82}.
\end{proof}

One of the most powerful techniques in the study of Fuchsian groups is the Riemann-Hurwitz formula relating the Euler characteristic of the whole group to that of a subgroup of finite index.

The concept of a rational Euler characteristic is extended to more general finitely presented groups. Further, these general rational Euler characteristics satisfy the Riemann-Hurwitz formula. For the general development of group homology and Euler characteristic we refer to \cite{29}.

\begin{thm} Let $G$ be a group of F-type. Then $G$ has a rational Euler characteristic $\chi(G)$ given by \[\chi(G)=2+\sum_{i=1}^na_i,\] where $a_i=-1$ if $e_i=0$ and $a_i=-1+\frac{1}{e_i}$ if $e_i\ge 2$. If $|G:H|<\infty$ then $\chi(G)$ is defined and $\chi(H)=|G:H|\cdot\chi(G)$. In addition, $G$ then is of finite homological type WFL, that is, $G$ is virtually torsion-free and for every torsion free subgroup of finite index $\mathbb{Z}$ admits a finite free resolution over the group ring $\mathbb{Z}G$, and $G$ has virtual cohomological dimension $\vcd(G)\le 2$.
\end{thm}

\begin{proof}Since $G$ is virtually finite we can apply the techniques of K.~Brown, see \cite{29}, to get an Euler characteristic for $G$ by the formula \[\chi(G)=\chi(G_1)+\chi(G_2)-\chi(A).\]
Recall that $\chi(H)$ is defined if $H$ is a free product of cyclics. Since $A$ is infinite cyclic we have $\chi(A)=0$. Therefore $\chi(G)=\chi(G_1)+\chi(G_2)$. $G_1$ and $G_2$ are free products of cyclic groups, so we can apply the computation rules for the free products, and we get \[\chi(G)=2+\sum_{i=1}^n\alpha_i\]
where $\alpha_i=-1$ if $e_i=0$ and $\alpha_i=-1+\frac{1}{e_i}$ if $e_i\ge 2$. We note that the Euler characteristic of a group of F-type can be zero. In fact, for instance, $\chi(G)=0$ if $n=2$. Now, from Section 7.6 in \cite{29}, we get that $G$ is of finite homological type WFL with $\vcd(G)\le 2.$
\end{proof}

\section{Additional Algebraic Results for Groups of F-Type}
The first two results follow by a straightforward application of the Nielsen cancellation method in free products with amalgamation, see \cite{49} and \cite{0} for a discussion of the Nielsen cancellation method.

\begin{thm}[Freiheitssatz for groups of F-type] Let $G$ be a group of F-type. Suppose that $UV$ involves all the generators. Then
\begin{enumerate}
\item[1.]Any subset of $(n-2)$-many of the given generators generates a free product of cyclics of the obvious orders.
\item[2.]If both $U$ and $V$ are proper powers in the respective factors $G_1$ and $G_2$, then any subset of $(n-1)$-many of the given generators generates a free product of cyclics of the obvious orders.
\end{enumerate}
\end{thm}

\begin{thm} Let $G$ be a group of F-type and let $H$ be a non-cyclic two-generator subgroup of $G$. Then $H$ is conjugate in $G$ to a subgroup $\langle x,y\rangle$ satisfying one of the following conditions:
\begin{enumerate}
\item[(1)]$\langle x,y\rangle$ is a free product of cyclic groups.
\item[(2)]$x^t$ is in $\langle U\rangle=\langle V\rangle$ for some natural number $t$ and $y^{-1}x^ty$ is in $\langle a_1,\dots, a_p\rangle$ or $y^{-1}x^ty$ is in $\langle a_{p+1},\dots , a_n\rangle$.
\end{enumerate}
\end{thm}

We could also prove Theorem 3.2 differently using \cite{128}.

Recall that a subgroup $K$ of $H$ is \textit{malnormal} in $H$ provided $gKg^{-1}\cap K=\{1\}$ unless $g\in H$. If $G$ is a group of F-type then $\langle U\rangle =\langle V\rangle$ is malnormal in $G$ if neither $U$ nor $V$ is a proper power or is conjugate to a word of the form $xy$ for elements $x,y$ of order 2.

Since non-cyclic two-generator subgroups of the free products of cyclic groups are free products of two cyclic groups we get the following corollary.

\begin{cor}\label{cor10a}Let $G$ be a group of F-type. Suppose further that $\langle U\rangle =\langle V\rangle$ is malnormal in $G$. Then any two-generator subgroup of $G$ is a free product of cyclics, and $\rank(G)\ge 3$.
\end{cor}

\begin{rem}In Subsection 3.5.4 of \cite{0} we just gave a list of properties of groups of F-type. Unfortunately, the last property is not correct as stated at this place. We just forgot to add that $\langle U\rangle =\langle V\rangle$ is malnormal in $G$. The correct statement is given in \cite[Corollary 1.5.20]{0}.

Indeed, the malnormality of $\langle U\rangle =\langle V\rangle$ in $G$ is important. For instance, the following groups of F-type \[\langle a_1,a_2,a_3,a_4\mid a_1^2=a_2^2=a_3^2=a_4^3=a_1a_2a_3a_4=1\rangle\]
and
\[\langle a_1,a_2,a_3\mid a_1^2=a_2^3=a_1a_2a_3^2=1\rangle\]
are two-generator groups. These examples are considered in \cite{0}.
\end{rem}

We may extend Corollary \ref{cor10a} easily to the following.

\begin{cor} Let $G$ be a group of F-type and $UV$ involve all the generators. Suppose further that $p\ge 3$, and $n-p\ge 3$, and that $\langle U\rangle = \langle V\rangle$ is malnormal in $G$. Then any three-generator subgroup of $G$ is a free product of cyclics.
\end{cor}

\begin{proof}Let $x,y,z\in G$. If $\langle x,y,z\rangle$ is already of rank two there is nothing to show. Hence, let $\rank(\langle x,y,z\rangle)=3$. Using the Nielsen cancellation method as described in Section 1.5 of \cite{0} we finally assume, without loss of generality, that one of the following two cases holds:

\begin{enumerate}
\item[a)] $x,y,z\in\langle a_1,\dots, a_p\rangle=G_1$,
\item[b)] $x,y\in\langle a_1,\dots, a_p\rangle =G_1$ with $z\notin G_2$.
\end{enumerate}
In case a) we are done. Now, let $x,y\in G_1$ and $z\notin G_2$. We consider the subgroup $\langle G_1, z\rangle$. We may assume that $z$ has the normal form \[z=\mu_1\mu_2\dots\mu_k,\] $k\ge 1$ and $\mu_1,\mu_k\in G_2\setminus\langle V\rangle$. But then $\langle y,x,z\rangle$ must be a free product of cyclics because $n-p\ge 3$. \end{proof}

If neither $U$ nor $V$ is a proper power, then from Theorem \ref{theo3}, $G$ is both residually finite and Hopfian. This holds in general for groups of F-type. Recall that a group $H$ is \textit{conjugacy separable} if given any non-trivial $g,h\in G$ that are not conjugate then there exists a finite quotient $H^\ast$ of $H$ where images of $g$ and $h$ are still not conjugate. Conjugacy separability implies residual finiteness.

From work of R.\,B.\,J.\,F.~Allenby in \cite{All} we get that a group of F-type is conjugacy separable. Allenby actually proves the following.

Let \[ G=\langle a_1,\dots, a_n\mid a_1^{e_1}=\dots=a_n^{e_n}=(UV)^m=1\rangle,\]
where $n\ge 2$, $e_i=0$ or $e_i\ge 2$, $1\le p\le n-1$, $U=U(a_1,\dots, a_p)$ a non-trivial cyclically reduced word in the free product on $a_1,\dots, a_p$, $V=V(a_{p+1},\dots,a_n)$ a non-trivial cyclically reduced word in the free product on $a_{p+1},\dots,a_n$, and $m\ge 2$. Then $G$ is conjugacy separable.

However, in this proof $m\ge 2$ is only used in the case where either $U$ or $V$ has finite order in the respective free product on the generators which they involve. In a group of F-type $U$ and $V$ are assumed to have infinite order so the result goes through. Hence, we have the following.

\begin{thm}A group of F-type is conjugacy separable and, hence, residually finite and Hopfian.
\end{thm}

A detailed proof can be found in \cite{82}. Recall that a group $H$ is \textit{subgroup separable} or \textit{LERF} (locally extended residually finite) if given any subgroup $K$ of $H$ and any element $g$ not in $K$ there exists a subgroup $K^\ast$ of finite index in $H$ such that $K$ is in $K^\ast$ and $g$ is not in $K^\ast$. From work of M.~Aab and G.~Rosenberger, see \cite{1}, we can deduce that groups of F-type are subgroup separable. Details can be found in \cite{0}.

\begin{thm}A group of F-type is subgroup separable.
\end{thm}

C.\,Y.~Tang, see \cite{Tang}, considered a class of groups which contains the groups of F-type.

\begin{thm} A group of F-type has solvable generalized word problem, and hence, solvable word problem.
\end{thm}

Recall that the generalized word problem is the following. Let $H=\langle X\mid R\rangle$ be a finite presentation and $W$ a finite set of words $\{w_1,\dots,w_m\}$ in $X$. Let $K$ be the subgroup of $G$ generated by $W$. Given a word $v$ in $X$, can one decide whether or not $v$ is in $K$? From Theorem 4.6 in \cite{152} we see that groups of F-type have solvable conjugacy problem. The applied technique also answers the \textit{power conjugacy problem} for groups of F-type, that is, given two elements to determine if a power of one is conjugate to a power of the other. Hence, we have the following.

\begin{thm}Groups of F-type have both solvable conjugacy problem and solvable power conjugacy problem.
\end{thm}

\section{Hyperbolic Groups of F-Type}
In the following we present the results by A. Juhász and G. Rosenberger, see \cite{117}, and consider the combinatorial curvature of a class of one-relator products which generalize groups of F-type. The main result is the following.

\begin{thm}\label{theoA} Let $G^{(1)}=\Asterisk_{\alpha\in \mathcal{T}_1}G_\alpha^{(1)}$, $G^{(2)}=\Asterisk_{\beta\in \mathcal{T}_2}G_\beta^{(2)}$, and let $G^{(0)}=G^{(1)}\ast G^{(2)}$. Let $U\in G^{(1)}$ and $V \in G^{(2)}$ be cyclically
reduced words of infinite order in $G^{(0)}$. Assume that if $U \in G_{\alpha}^{(1)}$, then $G_{\alpha}^{(1)}$ is free and $U$ is cyclically reduced in $G_{\alpha}^{(1)}$. Assume the analogue for $V$. Let $G$ be the quotient of $G^{(0)}$ by the normal closure of $UV$ in $G^{(0)}$. Then
\begin{enumerate}
\item[1.] $G$ has non-positive combinatorial curvature as a quotient of ${G}^{(0)}$.
\item[2.] $G$ is hyperbolic in the sense of M.~Gromov if and only if
$G^{(i)}$ is hyperbolic for $i=1,2$, and the following property holds:
\begin{enumerate}
\item[($\dagger$)] At least one of $U$ or $V$ is neither a proper power nor a product of two elements of order  2.
\end{enumerate}
\end{enumerate}
\end{thm}

Since groups of F-type are special cases of the groups mentioned in the theorem and cyclic groups are hyperbolic, we get the following corollary.

\begin{cor}Groups of F-type are hyperbolic unless $U$ is a proper power or a product of two elements of order $2$ and $V$ also is a proper power or a product of two elements of order $2$. In the last case they have non-positive combinatorial curvature. In particular, they satisfy a quadratic isoperimetric inequality.
\end{cor}

We consider van Kampen diagrams. All the unexplained terms concerning them can be found in \cite{27} and \cite{149}.

\begin{rem}
\begin{enumerate}
\item[1.] Let $M$ be a diagram over a free group or free product $F$. We shall denote by $|\partial M |$ the length of a cyclically reduced label of $M$ over $F$.
\item[2.] Assume $A$ is given by a presentation $\langle X \mid R\rangle$. Then we shall always assume that $R$ is cyclically reduced. Also we shall assume that our van Kampen diagram contains a minimal number of regions for a simply connected given boundary label.
\item[3.] Recall from \cite{101} that a presentation $\langle X\mid R\rangle$ has non-positive (negative) combinatorial curvature if every inner region $D$ of every van Kampen diagram over $R$ satisfies that the excess $\kappa(D)=2 \pi+\sum_{i=1}^{n}\left(\theta_{i}-\pi\right)$ is non-positive (negative). Here, $\partial D=v_{1} e_{1} v_{2} e_{2} \cdots v_{n} e_{n} v_{1}$, $v_{i}$ vertices on $\partial D$, $e_{i}$ edges on $\partial D$, and $\theta_{i}$ are the inner angles of the polygon $D$ at $v_i$. In general, one takes $\theta_{i}=\frac{2 \pi}{d(v_{i})}$ where $d\left(v_{i}\right)$ is the valency of $v_{i},$ that is, distribute the curvature equally among the regions containing $v_i$

However, from the point of view of the theory of groups with non-positive curvature, see \cite{27} and the theory of hyperbolic groups, see \cite{101}, it is immaterial how the angles around $v_i$ are distributed, as long as they sum up to $2 \pi$ and $\theta_{i}-\frac{2 \pi}{d(v_{i})}<\frac{2 \pi}{6}.$ In the following we shall make use of this remark.
\end{enumerate}
\end{rem}
The following lemma is an immediate consequence of the construction of diagrams over free products. We omit its proof.

\begin{lem}Let $A=\langle X \mid R\rangle$, $B=\langle Y \mid S\rangle$ be finitely generated, and let $G=A \ast B / D,$ where $D$ is the normal closure of $T \subset A\ast B$. Let $Q= R\cup S\cup T$. Thus $Q \subset  F(X\cup Y)$. Assume that there are constants $C$ and $a$ such that the following isoperimetric inequalities hold for $R$-diagrams $M$, $S$-diagrams $N$ and $T$-diagrams $H$ respectively:
$$\mathrm{Vol}(M) \leq C |\partial M|^a, \mathrm{Vol}(N) \leq C|\partial N|^a \text{ and } \mathrm{Vol}(H) \leq C|\partial H|^a$$ where for a $P$-diagram $Z$, $\mathrm{Vol}(Z)$ is the number of regions of $Z$.

Then there is a constant $C' \geq C$ depending on $Q$ such that for every word $W\in F(X\cup Y)$ which represents $1$ in $G$ there is a $Q$-diagram $U$ with boundary label $W$ such that $\mathrm{Vol}(U)\leq C'|W|^a$.
\end{lem}

We also mention the following corollary.

\begin{cor}
\begin{enumerate}
\item[1.]If $A$ and $B$ are hyperbolic and the presentation of $G$ as a quotient of the free product $A\ast B$ satisfies a linear isoperimetric inequality, then $G$ is hyperbolic.
\item[2.]If $A$ and $B$ satisfy a polynomial isoperimetric inequality of degree $k$ and $G$ has nonpositive combinatorial curvature as a quotient of $A\ast B$, then $G$ satisfies a polynomial isoperimetric inequality of degree $\max (k, 2)$.
\end{enumerate}
\end{cor}

Let $R$ be the set of all the cyclically reduced cyclic conjugates of $UV$ and $(UV)^{-1}$. We now describe the structure of the van Kampen diagrams for the groups given in Theorem \ref{theoA}. Let $M$ be a reduced $R$-diagram, and let $D$ be an inner region of $M$. Then $D$ has a boundary cycle $v_{0} \mu v_{1} \nu v_{0}$ such that the label of $\mu$ is $U$ and the label of $v$ is $V$. From now on, we shall not distinguish between a path in $M$ and its label in $G^{(0)}$. Since the vertices $v_{0}$ and $v_{1}$ separate $\mu$ from $v$ we shall call them separating vertices, see Figure \ref{JR1}.

\begin{figure}[h!]\label{JR1}
   \centering
   
\begin{tikzpicture}

\node (v3) at (-3,0) {\textbullet};
\node (v1) at (-9,0) {\textbullet};
\node (v2) at (-6,1.5) {};
\node (v4) at (-6,-1.5) {};
\draw[ decoration={markings, mark=at position 0.625 with {\arrow{>}}},
        postaction={decorate}]  plot[smooth, tension=.7] coordinates {(v1) (v2) (v3)};
\draw[ decoration={markings, mark=at position 0.625 with {\arrow{<}}},
        postaction={decorate}]  plot[smooth, tension=.7] coordinates {(-9,0) (v4) (-3,0)};
\node at (-6,0) {$D$};
\node at (-9.5,0) {$v_0$};
\node at (-2.5,0) {$v_1$};
\node at (-6,2) {$\mu$};
\node at (-6,-2) {$\nu$};
\end{tikzpicture}

         \caption{}
         \label{JR1}        
\end{figure}

\begin{thm}\label{JRP1} Separating vertices of D have valency at least $4$.
\end{thm} 

\begin{proof}
Assume $v_0$ has valency $3$. Then there are regions $D_1$ and $D_2$ which contain $v_0$ on their boundary. Let $x$ be the first letter on the edge common to $D_1$ and $D_2$ which contains $v_{0}$. If $x \in G^{(1)}$, then $v_{0}$ is a separating vertex of $D_{2}$, see Figure \ref{JR2} (A).  Consequently a non-trivial tail of $\nu$ is a common edge of $D$ and $D_2.$ But then, since $U$ is cyclically reduced and $U \neq U^{-1}$, $D$ completely cancels $D_2$. This violates the assumption that $M$ is reduced.

\begin{figure}[h!]{}
     \centering
     \begin{subfigure}[]{0.49\textwidth}
         \centering
\begin{tikzpicture}[scale=0.8]
e
\node (v1) at (0,0) {\textbullet};
\node (v2) at (-4,0) {};

\draw  plot[smooth, tension=.7] coordinates {(v1) (v2)};
\draw  plot[smooth, tension=.7] coordinates {(2,1) (-1.4,2) (v2)};
\draw  plot[smooth, tension=.7] coordinates {(v2) (-1.4,-2) (2,-1)};
\draw  plot[smooth, tension=.7] coordinates {(v1)};
\draw  plot[smooth, tension=.7] coordinates {(v1) (2,1) (4,0) (2,-1) (0,0)};
\node at (-1,1) {$D_1$};
\node at (-1,-1) {$D_2$};
\node at (0.5,0) {$v_0$};
\node at (-2,0.3) {$x$};
\node at (3,0) {$D$};
\end{tikzpicture}
         \caption{}
         \label{JR2A}
     \end{subfigure}
     \hfill
     \begin{subfigure}[]{0.49\textwidth}
         \centering
 \begin{tikzpicture}[scale=0.8]
e
\node (v1) at (0,0) {\textbullet};
\node (v2) at (-4,0) {};

\draw  plot[smooth, tension=.7] coordinates {(v1) (v2)};
\draw  plot[smooth, tension=.7] coordinates {(2,1) (-1.4,2) (v2)};
\draw  plot[smooth, tension=.7] coordinates {(-2.5,0) (-1.4,-2) (2,-1)};
\draw  plot[smooth, tension=.7] coordinates {(v1)};
\draw  plot[smooth, tension=.7] coordinates {(v1) (2,1) (4,0) (2,-1) (0,0)};
\node at (-1,1) {$D_1$};
\node at (-1,-1) {$D_2$};
\node at (0.5,0) {$v_0$};
\node at (-2,0.3) {$x$};
\node at (3,0) {$D$};
\end{tikzpicture}
         \caption{}
         \label{JR2b}
     \end{subfigure}
     \caption{}\label{JR2}
\end{figure}

Similarly, if $x \in G^{(2)}$, then $v_{0}$ is a separating vertex of  $D_1$, by symmetry, see Figure \ref{JR2} (B), leading to the same contradiction. Thus $v_0$ has valency $\neq 3$. 

Finally, assume $v_0$ has valency $2$. Recall that $U$ and $V$ are cyclically reduced. Let first $U$ and $V$ have length at least $2$. Then $v_0$ is a seperating vertex of $D_1$ and $D_1$ cancels $D$ completely as, for instance, in Figure \ref{JR3}. Now let $U$ or $V$ have length $1$. If, for instance, $U = x$ and $v_0$ has valency $2$ then $x = x^{-1}$, that is, $U = U^{-1}$. This contradicts that $U$ has infinite order.\end{proof}

\begin{figure}[h!]\label{JR3}
   \centering
   \begin{tikzpicture}

\node (v3) at (-3,0) {};
\node (v1) at (-9,0) {};
\node (v2) at (-6,1.5) {};
\node (v4) at (-6,-1.5) {};
\draw[decoration={markings, mark=at position 0.625 with {\arrow{>}}},
        postaction={decorate}]  plot[smooth, tension=.7] coordinates {(v1) (v2) (v3)};
\draw[decoration={markings, mark=at position 0.625 with {\arrow{>}}},
        postaction={decorate}]  plot[smooth, tension=.7] coordinates {(-9,0) (v4) (-3,0)};

\node at (-6,2) {$u$};
\node at (-6,-2) {$u$};
\draw[decoration={markings, mark=at position 0.625 with {\arrow{>}}},
        postaction={decorate}]  plot[smooth, tension=.7] coordinates {(-3,0) (-9,0)};
\node at (-6,0.7) {$D$};
\node at (-6,-0.7) {$D_2$};
\end{tikzpicture}

         \caption{}
         \label{JR3}
\end{figure}

\begin{cor}
If $M$ satisfies the small cancellation condition $C(5)$, then $M$ has a non-positive combinatorial curvature.
\end{cor}

The inner region $D$ always has two vertices with valency $\geq 4$ (namely $v_0$ and $v_1$).

\begin{rem}Neither $U$ nor $V$ can be a piece, for then as we have seen, the diagram is not reduced. Consequently, $U$ and $V$ are each the product of at least two pieces.

Since by the theorem no region $D_{1}$ can be a neighbor of $D$ with a common edge, see Figure \ref{JR4}, which contains $v_{0}$ as an inner vertex, $D$ has at least $4$ neighbors. 
\end{rem}

\begin{figure}[h!]\label{JR4}
   \centering
\begin{tikzpicture}
\draw  plot[smooth, tension=.7] coordinates {(-0.5,0) (0,0.5) (1.5,0.5) (2.5,0) (1.5,-0.5) (0,-0.5) (-0.5,0)};
\node at (-0.5,0) {\textbullet};
\node at (0,0) {$v_0$};
\node at (-1.5,0) {$D_1$};
\node at (1,0) {$D$};
\draw  plot[smooth, tension=.7] coordinates {(0,0.5)};
\draw  plot[smooth, tension=.7] coordinates {(0,0.5) (-1.5,1) (-2.5,0) (-1.5,-1) (0,-0.5)};
\end{tikzpicture}
         \caption{}
         \label{JR4}
\end{figure}

This proves the following.

\begin{cor}
\begin{enumerate}
\item[1.] $M$ satisfies the small cancellation condition $C(4)$.
\item[2.] An inner region $D$ has $4$ neighbors if and only if its boundary cycle is subdivided into $4$ edges in such a way that the label on the first two edges is $U$ and the label on the second two edges is $V$. In particular, if such a region exists, then $U$ and $V$ are each the product of two pieces.
\end{enumerate}
\end{cor}

Thus, in what follows, we shall study the situation when $U$ is the product of two pieces. Denote by $\mathcal{D}_{0}$ the set of all the inner regions in $M$ which have a boundary path with label in $\{U^{\pm 1}, V^{\pm 1}\}$ and have a decomposition to two pieces by a vertex which is not a separating vertex.

Denote the set of all the inner vertices of $M$ which are not separating vertices and divide a boundary path of a region $D$ with label $U^{\pm 1}$ or label $V^{\pm 1}$ to two parts by $\mathcal{V}_0$. For $i \geq 3$ denote by $\mathcal{V}_i$ the set of all the vertices of $M$ not in $V_{0}$ which have valency $i$.

Denote by $\kappa(U)$ the contribution of all the vertices of an inner region $D$ which are inner vertices of the boundary path $\mu$ with label $U$ to the excess of $D$, and let $\kappa(V)$ be defined analogously. Then
\begin{align}
\kappa(D)=2 \pi+\kappa(U)+\kappa(V)+(\theta_{0}-\pi)+(\theta_{t}-\pi)=\kappa(U)+\kappa(V)+\theta_{0}+\theta_{t}
\end{align}

where $\theta_{0}$ and $\theta_{t}$ denote the angles of the two separating vertices $v_{0}$ and $v_{t}$, compare Figure \ref{JR1}.

\begin{defi}
\begin{enumerate}
\item[1.]A \textit{corner} consists of a pair $(e, e^{1})$ of edges which have the same initial vertex and are such that the path $e^{-1} e^{1}$ is a subpath of a boundary cycle of a region associated with the corner.
\item[2.]Given a vertex $u \in \mathcal{V}_{0},$ a \textit{bad} corner at $u$ is a corner $(e, e^{1})$ such that the path $e^{-1} e^{1}$ has label one of $U^{\pm 1}, V^{\pm 1}$. A \textit{good} corner is a corner which is not bad.
\item[3.]Call two corners \textit{adjacent} if they have an edge in common.
\end{enumerate}
\end{defi}

\begin{thm}\label{JRP2} Let the edge pair $\left(\alpha^{-1}, \beta\right)$ define a bad corner at the vertex $v$ of the region $D$. Then
\begin{enumerate}
\item[1.]The two corners at $v$ which are adjacent to the given corner are good corners.
\item[2.]The vertex $v$ is preceded or followed (in the boundary cycle of $D$) by a vertex of valency 4 which is a separating vertex of $D$.
\end{enumerate}
\end{thm}

\begin{figure}[h!]\label{JR5}
   \centering
\begin{tikzpicture}[scale=0.93]

\draw  plot[smooth, tension=.7] coordinates {(-9.6,2) (-10.8,4.4)};
\draw  plot[smooth, tension=.7] coordinates {(-7.2,2.4) (-8.4,4.8)};
\draw  plot[smooth, tension=.7] coordinates {(-10.8,4.4) (-9.6,5.2) (-8.4,4.8)};
\draw  plot[smooth, tension=.7] coordinates {(-7.2,2.4)};
\draw  plot[smooth, tension=.7] coordinates {(-7.2,2.4) (-6.8,4)};
\draw  plot[smooth, tension=.7] coordinates {(-4.8,2) (-4.4,3.6)};
\draw  plot[smooth, tension=.7] coordinates {(-6.8,4) (-5.6,4.4) (-4.4,3.6)};
\draw  plot[smooth, tension=.7] coordinates {(-7.2,2.4)};
\draw[ decoration={markings, mark=at position 0.92 with {\arrow{<}}},
        postaction={decorate}]  plot[smooth, tension=.7] coordinates {(-7.2,2.4) (-9.6,2) (-10.8,0) (-8.4,-2) (-5.2,-1.6) (-3.6,0) (-4.8,2) (-7.2,2.4)};
\draw[ decoration={markings, mark=at position 0.95 with {\arrow{>}}},
        postaction={decorate}]  plot[smooth, tension=.7] coordinates {(-7.2,2.4)  (-4.8,2) (-3.6,0)  (-5.2,-1.6) (-8.4,-2) (-10.8,0) (-9.6,2)(-7.2,2.4)};
\node at (-10.8,0) {\textbullet};
\node at (-3.6,0) {\textbullet};
\node at (-10.8,4.4) {\textbullet};
\node at (-8.4,4.8) {\textbullet};
\node at (-6.8,4) {\textbullet};
\node at (-4.4,3.6) {\textbullet};
\node at (-7.2,2.4) {\textbullet};
\node at (-7.5,0.4) {$D$};
\node at (-9.3,3.8) {$D_1$};
\node at (-5.7,3.4) {$D_2$};
\node at (-11.3,-0.1) {\small $v_0$};
\node at (-3.2,0) {\small $v_1$};
\node at (-11.3,4.4) {\small $w_0$};
\node at (-8.1,4.9) {\small $w_1$};
\node at (-7.3,4.1) {\small $w_2$};
\node at (-4,3.6) {\small $w_3$};
\node at (-9.6,5.5) {\small $\nu_1$};
\node at (-5.5,4.7) {\small $\nu_2$};
\node at (-10.3,2.7) {\small $\mu_1$};
\node at (-7.2,3.3) {\small $\mu_2$};
\node at (-8.4,2.1) {\small $\alpha$};
\node at (-5.8,2) {\small $\beta$};
\node at (-7.2,2.2) {\small $v$};
\end{tikzpicture}
         \caption{}
         \label{JR5}
\end{figure}

\begin{proof}
Let $w_{0}$ and $w'$ be the separating vertices of $F$, see Figure \ref{JR5}. Then $w_{1} \neq v$ since $v$ is not a separating vertex by assumption and $w_{0} \neq v_{0}$, otherwise either $D_{1}$ cancels $D$ (if $\mu$ has label $U$) or $U$ is not cyclically reduced (if $\mu_{1}$ has label $U^{-1}$).\end{proof}

Now define the angles in $M$. If $v \in \mathcal{V}_{i}$, $i \geq 3$, then assign the value $\frac{2 \pi}{i}$ to every angle having $v$ as its vertex. If $v \in \mathcal{V}_0$ and $d=d(v) \geq 5$ then assign to each angle having $v$ as its vertex the value $\frac{2 \pi}{d}$. Let $\varepsilon=\frac{2 \pi}{8 m^{2}}$, $m=2(|U|+|V|)$. If $v \in \mathcal{V}_{0}$ and $d(v)=4$, then by Theorem \ref{JRP2} there are at most two regions which contain $v$ on their boundary and belong to $\mathcal{D}_{0}$ ($D_{1}$ and $D_{3}$ in Figure \ref{JR6} (A) and (B)).

If $D_1$ and $D_3$ are in $\mathcal{D}_{0},$ then assign $\frac{2 \pi}{4}-\varepsilon$ to $\gamma_{1}$ and $\gamma_{3}$ and assign $\frac{2 \pi}{4}+\varepsilon$ to $\gamma_{2}$ and $\gamma_{4}$. If $D_{1} \in \mathcal{D}_{0}$ and $D_{3} \notin \mathcal{D}_{0},$ then assign $\frac{2 \pi}{4}-\varepsilon$ to $\gamma_{1}$, $\frac{2 \pi}{4}+\frac{\varepsilon}{2}$ to $\gamma_{2}$ and $\gamma_{4}$ and $\frac{2 \pi}{4}$ to $\gamma_{3}$. Finally, if $d(v)=3$, then by Theorem \ref{JRP2} exactly one of the regions containing $v$, say $D_1$, belongs to $\mathcal{D}_{0}$, see Figure \ref{JR6} (B). Assign $\frac{2 \pi}{4}-\varepsilon$ to $\gamma_{1}$ and assign $\frac{2 \pi}{3}+\frac{2 \pi}{24}+\frac{\varepsilon}{2}$ to $\gamma_{2}$ and to $\gamma_{3}$.

\begin{figure}[h!]{}
     \centering
     \begin{subfigure}[]{0.49\textwidth}
         \centering
\begin{tikzpicture}[scale=0.90]
\node (v1) at (-2.5,-1.5) {};
\node (v2) at (-0.5,0.5) {};
\node (v3) at (3.5,0) {};

\draw  plot[smooth, tension=.7] coordinates {(-2.5,-1.5) (-4,-0.5)};
\draw  plot[smooth, tension=.7] coordinates {(v2) (-2,1.5)};
\draw  plot[smooth, tension=.7] coordinates {(-2,1.5) (-3.5,1) (-4,-0.5)};
\draw  plot[smooth, tension=.7] coordinates {(v2) (0,2.5)};
\draw  plot[smooth, tension=.7] coordinates {(v3) (4.5,2)};
\draw  plot[smooth, tension=.7] coordinates {(0,2.5) (2,3) (4.5,2)};
\draw  plot[smooth, tension=.7] coordinates {(-1.5,4) (-2,1.5)};
\draw  plot[smooth, tension=.7] coordinates {(-1.5,4) (0,2.5)};
\node (v4) at (-0.5,0.5) {\textbullet};
\node at (-2.5,0.5) {$D_2$};
\node at (-1,2.5) {$D_3$};
\node at (1.5,2) {$D_4$};
\node at (0.5,-1.5) {$D_1$};
\node at (-0.1,0.8) {\tiny$\gamma_4$};
\node at (-1,0.5) {\tiny$\gamma_2$};
\node at (-0.6,0.1) {$v$};
\draw  plot[smooth, tension=.7] coordinates {(-2.5,-1.5) (v4) (v3)};
\draw  plot[smooth, tension=.7] coordinates {(-2.5,-1.5) (0,-3) (3,-2.5) (v3)};
\node at (-0.3,0.2) {\tiny$\gamma_1$};
\node at (-0.7,0.9) {\tiny$\gamma_3$};
\end{tikzpicture}
         \caption{}
         \label{JR6a}
     \end{subfigure}
 \hfill    
     \begin{subfigure}[]{0.49\textwidth}
         \centering
         \begin{tikzpicture}[scale=0.90]

\node (v1) at (-2.5,-1.5) {};
\node (v2) at (-0.5,0.5) {};
\node (v3) at (3.5,0) {};
\node (v8) at (0,4) {};

\node (v4) at (-0.5,0.5) {\textbullet};
\node at (-2,0.5) {$D_2$};

\node at (1,1.7) {$D_3$};
\node at (0.5,-1.5) {$D_1$};
\node at (-0.35,0.8) {\tiny$\gamma_3$};
\node at (-1,0.5) {\tiny$\gamma_2$};
\node at (-0.6,0.1) {$v$};
\draw  plot[smooth, tension=.7] coordinates {(-2.5,-1.5) (v4) (v3)};
\draw  plot[smooth, tension=.7] coordinates {(-2.5,-1.5) (0,-3) (3,-2.5) (v3)};
\draw  plot[smooth, tension=.7] coordinates {(v4) (-1.5,2)};
\draw  plot[smooth, tension=.7] coordinates {(v3) (3,2)};
\draw  plot[smooth, tension=.7] coordinates {(-1.5,2) (0.5,3) (3,2)};
\draw  plot[smooth, tension=.7] coordinates {(-2.5,-1.5) (-3.5,0)};
\draw  plot[smooth, tension=.7] coordinates {(-3.5,0) (-3,1) (-2,1.5) (-1.17,1.5)};
\node at (-0.3,0.2) {\tiny$\gamma_1$};
\end{tikzpicture}
         \caption{}
         \label{JR6b}
     \end{subfigure}
     \caption{}\label{JR6}
\end{figure}

\begin{thm}\label{JRP3} Let $D$ be an inner region of $M$. If $D$ has a vertex $v \in \mathcal{V}_{0}$ on its boundary, then $\kappa(D)<0$.
\end{thm}

\begin{proof}
Denote by $a_i$ the number of inner vertices of $\mu$ (see Figure \ref{JR1}) in $\mathcal{V}_i$, $i\ge 3$, and let $a_4^*=\sum_{i\ge 4}a_i$. For $i\ge 5$, denote by $b_i$ the number of inner vertices of $\mu$ in $\mathcal V_0$ with valency $i$. Define numbers $b_4',b_4'', c_4$ and $c_3$ as follows: 
\begin{itemize}
\item $b_4'$ is the number of inner vertices of $\mu$ in $\mathcal{V}_0$ which are of degree $4$ and such that
\begin{enumerate}
\item[a)]the corner associated with $D$ is a good corner;
\item[b)]exactly one of the two corners adjacent to the corner associated with $D$ is a bad corner;
\end{enumerate}

\item $b_4''$ is the number of inner vertices of $\mu$ in $\mathcal{V}_{0}$ which are of degree $4$ and such that
\begin{enumerate}
\item[a)]the corner associated with $D$ is a good corner;
\item[b)]both the two corners adjacent to the corner associated with $D$ are bad corners;
\end{enumerate}
\item $c_{4}$ is the number of inner vertices of $\mu$ in $\mathcal{V}_{0}$ which are of degree $4$ and such that
\begin{enumerate}
\item[a)]the corner associated with $D$ is a good corner;
\item[b)]neither of the two corners adjacent to the corner associated with $D$ are bad corners (so that the remaining corner which is 'opposite' to the corner associated with $D$, must be a bad corner);
\end{enumerate}

\item $c_{3}$ is the number of inner vertices of $\mu$ in $\mathcal{V}_{0}$ of degree $3$ such that the corner associated with $D$ is a good corner (in which case exactly one of the other two corners adjacent to the corner associated with $D$ is a bad corner).\end{itemize} Assume $D \notin \mathcal{D}_{0}$ and evaluate $\kappa(U)$. Thus,
$$
\begin{aligned}
\kappa(U) &=\sum_{i \geq 3} a_{i}\left(\frac{2 \pi}{i}-\pi\right)+\sum_{i \geq 5} b_{i}\left(\frac{2 \pi}{i}-\pi\right)+b_{4}^{\prime}\left(\frac{2 \pi}{4}+\varepsilon-\pi\right) \\
&+b_{4}^{\prime \prime}\left(\frac{2 \pi}{4}+\frac{\varepsilon}{2}-\pi\right)+c_{4}\left(\frac{2 \pi}{4}-\pi\right)+c_{3}\left(\frac{2 \pi}{3}+\frac{2 \pi}{24}-\pi+\frac{\varepsilon}{2}\right) \\
& \leq-\frac{\pi}{4} c_{3}-\frac{\pi}{2} a_{4}^{*}-\frac{\pi}{2}\left(b_{4}^{\prime}+b_{4}^{\prime \prime}\right)+\frac{\varepsilon}{2}\left(2 b_{4}^{\prime}+b_{4}^{\prime}+c_{3}\right) \\
& \leq-\frac{\pi}{4} c_{3}-\frac{\pi}{2} a_{4}^{*}-\frac{\pi}{2}\left(b_{4}^{\prime}+b_{4}^{\prime \prime}\right)+\frac{\pi}{8 m^{2}}(2 m).
\end{aligned}
$$

Assume $c_{3}>0 .$ Then, by Theorem \ref{JRP2}, $a_{4}^{*} \geq \frac{1}{2} c_{3}$, hence $-\frac{\pi}{2}\left(a_{4}^{*}+\frac{1}{2} c_{3}\right) \leq-\frac{\pi}{2}-\frac{\pi}{4},$ and
$$
\kappa(U) \leq-\left(\frac{\pi}{4}+\frac{\pi}{2}\right)+\frac{\pi}{8 m^{2}}(2 m)=\frac{3}{4} \pi+\frac{\pi}{4 m}<-\frac{3}{4} \pi+\frac{\pi}{4}=-\frac{\pi}{2}
$$
as $m>1 .$ Thus
\begin{align}
\text{if } D \notin D_{0} \text{ and } c_{3}>0\text{, then }\kappa(U)<-\frac{\pi}{2}.
\end{align}

Assume $c_{3}=0 .$ Then $b_{4}^{\prime}+b_{4}^{\prime \prime} \neq 0,$ hence, by Theorem \ref{JRP2}, $a_{4}>0$. Consequently,

\[k(U) \leq-\frac{\pi}{2} a_{4}-\frac{\pi}{2}\left(b^{\prime}+b^{\prime \prime}\right)+\frac{\pi}{4 m}  \leq-\frac{\pi}{2}-\frac{\pi}{2}+\frac{\pi}{4 m} 
<\frac{\pi}{2}-\frac{\pi}{2}+\frac{\pi}{4}<\frac{\pi}{2}.\]

\begin{align}
\text{If } D \notin \mathcal{D}_0 \text{ and } c_{3}=0\text{, then }\kappa(U)<-\frac{\pi}{2}. 
\end{align}
Assume now that $D \in \mathcal{D}_{0}$. Then, by the choice of $\gamma_1$, $\kappa(U)=-\frac{\pi}{2}-\varepsilon$.

\begin{align}
\text{If } D \in \mathcal{D}_{0}\text{, then } \kappa(U)<-\frac{\pi}{2}.
\end{align}

Evaluate now $\kappa(D)$. By (4.1), (4.2), (4.3), and (4.4) we have \[\kappa(D)=\kappa(U)+\kappa(V)+\frac{\pi}{2}+\frac{\pi}{2}<-\frac{\pi}{2}-\frac{\pi}{2}+\frac{\pi}{2}+\frac{\pi}{2}=0\] which yields the theorem.
\end{proof}

\begin{cor} $D$ has non-positive combinatorial curvature. $D$ has zero curvature if and only if $D$ has four neighbors and every boundary vertex is a separating vertex with valency $4$.
\end{cor}
We now prove Theorem \ref{theoA}.

\begin{proof}[Proof of Theorem \ref{theoA}] We remark that we already have the first statement in Theorem \ref{theoA} by the previous results. We now consider the relation ($\dagger$). Let the notation be as in Figures \ref{JR7} (A) and (B). Then either $v$ is a separating vertex for $D_1$ or $v$ is a separating vertex for $D_2$.

\begin{figure}[h]{}
     \centering
     \begin{subfigure}[]{0.49\textwidth}
         \centering
\begin{tikzpicture}[scale=0.85]

\draw[decoration={markings, mark=at position 0.1 with {\arrow{<}}},
        postaction={decorate},decoration={markings, mark=at position 0.4 with {\arrow{<}}},
        postaction={decorate}]  (-3,1) ellipse (2.5 and 1);
\node (v1) at (-3,2) {\textbullet};
\draw  plot[smooth, tension=.7] coordinates {(-5.5,1) (-7,2.5)};
\draw  plot[smooth, tension=.7] coordinates {(v1) (-4.5,3) (-7,2.5)};
\draw  plot[smooth, tension=.7] coordinates {(-0.5,1) (1,2.5)};
\draw  plot[smooth, tension=.7] coordinates {(1,2.5) (-1.5,3) (v1)};
\draw  plot[smooth, tension=.7] coordinates {(-4.5,3) (-3,3.8) (-1.5,3)};

\node at (-3,1) {$D$};
\node at (-4.5,1.5) {$A$};
\node at (-3,1.7) {$v$};
\node at (-1.5,1.5) {$B$};
\node at (-1,2.5) {$D_2$};
\node at (-5,2.5) {$D_1$};
\node at (-3,3) {$D_3$};
\node at (-6,3.2) {$V$};
\node at (0,3.2) {$U$};
\end{tikzpicture}
         \caption{}
         \label{JR7a}
     \end{subfigure}
     \hfill
     \begin{subfigure}[]{0.49\textwidth}
         \centering
\begin{tikzpicture}[scale=0.85]

\draw[decoration={markings, mark=at position 0.1 with {\arrow{<}}},
        postaction={decorate},decoration={markings, mark=at position 0.4 with {\arrow{<}}},
        postaction={decorate}]  (-3,1) ellipse (2.5 and 1);
\node (v1) at (-3,2) {\textbullet};
\draw  plot[smooth, tension=.7] coordinates {(-5.5,1) (-7,2.5)};
\draw[decoration={markings, mark=at position 0.7 with {\arrow{<}}},
        postaction={decorate}]   plot[smooth, tension=.7] coordinates {(v1) (-4.5,3) (-7,2.5)};
\draw plot[smooth, tension=.7] coordinates {(-0.5,1) (1,2.5)};
\draw[decoration={markings, mark=at position 0.6 with {\arrow{<}}},
        postaction={decorate}]  plot[smooth, tension=.7] coordinates {(1,2.5) (-1.5,3) (v1)};

\node at (-3,1) {$D$};
\node at (-4.5,1.5) {$A$};
\node at (-3,1.7) {$v$};
\node at (-1.5,1.5) {$B$};
\node at (-6,3.2) {$V$};
\node at (0,3.2) {$V$};
\end{tikzpicture}
         \caption{}
         \label{JR7b}
     \end{subfigure}
     \caption{}\label{JR7}
\end{figure}

If $v$ is a separating vertex for only one of $D_{1}$ and $D_{2},$ then it is a separating vertex for $D_3$ and hence for two neighboring regions of $D$ ($D_1$ and $D_3$ or $D_2$ and $D_3$). But the boundary labels of such neighbors cancel each other, contradicting our assumption that $M$ contains a minimal number of regions with the given boundary label. Consequently, $v$ is a separating vertex for both $D_{1}$ and $D_{2}$, see Figure \ref{JR7} (B).

Let $A$ be the label of the piece common to $D$ and $D_1,$ and let $B$ be the label of the piece common to $D$ and $D_2$. Then one of the following equations holds: $(AB)(AB)=U(AB)V$ (see Figure \ref{JR8} (A)) or $A=A^{-1}$ and $B=B^{-1}$, see Figure \ref{JR8} (B). In the first case $A=D^{\alpha}$ and $B=D^{\beta}$ and in the second case $A$ and $B$ have order $2$ as required.

\begin{figure}[h]{}
     \centering
     \begin{subfigure}[]{0.49\textwidth}
         \centering
\begin{tikzpicture}[scale=0.75]

\draw[decoration={markings, mark=at position 0.4 with {\arrow{<}}},
decoration={markings, mark=at position 0.16 with {\arrow{<}}},
        postaction={decorate}]  (-3.5,0.5) ellipse (2.5 and 1);
\draw[decoration={markings, mark=at position 0.15 with {\arrow{<}}},
decoration={markings, mark=at position 0.4 with {\arrow{<}}},
        postaction={decorate}]  (-8.5,0.5) ellipse (2.5 and 1);
\node at (-6,0.5) {\textbullet};
\node at (-3.5,1.5) {\textbullet};
\node at (-8.5,1.5) {\textbullet};
\draw[decoration={markings, mark=at position 0.15 with {\arrow{<}}},
decoration={markings, mark=at position 0.85 with {\arrow{<}}},
        postaction={decorate}]  plot[smooth, tension=.7] coordinates {(-6,1.5) (-8.5,2.5) (-6,3.5) (-3.5,2.5) (-6,1.5)};

\node at (-10,1) {$A$};

\node at (-7.5,1) {$B$};
\node at (-5,1) {$A$};
\node at (-2.5,1) {$B$};
\node at (-7.5,2.5) {$A$};

\node at (-4.5,2.5) {$B$};
\end{tikzpicture}
         \caption{}
         \label{JR8a}
     \end{subfigure}
     \hfill
     \begin{subfigure}[]{0.49\textwidth}
         \centering
\begin{tikzpicture}[scale=0.75]

\draw[decoration={markings, mark=at position 0.4 with {\arrow{<}}},
        postaction={decorate}]  (-3.5,0.5) ellipse (2.5 and 1);
\draw[decoration={markings, mark=at position 0.15 with {\arrow{<}}},
        postaction={decorate}]  (-8.5,0.5) ellipse (2.5 and 1);
\node at (-6,0.5) {\textbullet};
\node at (-3.5,1.5) {\textbullet};
\node at (-8.5,1.5) {\textbullet};
\draw[decoration={markings, mark=at position 0.15 with {\arrow{>}}},
decoration={markings, mark=at position 0.85 with {\arrow{>}}},
        postaction={decorate}]  plot[smooth, tension=.7] coordinates {(-6,1.5) (-8.5,2.5) (-6,3.5) (-3.5,2.5) (-6,1.5)};

\node at (-7.5,1) {$B$};
\node at (-5,1) {$A$};
\node at (-7.5,2.5) {$B$};

\node at (-4.5,2.5) {$A$};
\end{tikzpicture}
         \caption{}
         \label{JR8b}
     \end{subfigure}
     \caption{}\label{JR8}
\end{figure}

Assume ($\dagger$) holds. Then by Theorem \ref{JRP3} the combinatorial curvature at each inner region of a derived van Kampen diagram over $R$ is strictly negative. Therefore $G$ is hyperbolic.

Assume now that ($\dagger$) does not hold. Then $UV$ (or $VU$) has one of the following forms.
\begin{enumerate}
\item[a)]$A^{n} B^{m}$ for $n, m \geq 2$;
\item[b)]$ABCD, A=A^{-1}$, $B=B^{-1}$, $C=^{-1}$, $D=D^{-1}$;
\item[c)]$ABC^n$, $n\ge 2$, $A=A^{-1}$, $B=B^{-1}$.
\end{enumerate}

In case a) let $H=\langle A, B\rangle$. Then it follows from the fact that every van Kampen diagram over the symmetric closure of $UV$ has a non-positive combinatorial curvature that \[H \cong\langle a,b \mid a^n=b^{-m}\rangle.\]

Consequently, $G$ is not hyperbolic.

A similar argument shows that the subgroup of G generated by $A,B,C$ and $D$ in case b) is isomorphic to $\langle a, b, c, d\mid  a^{2}, b^{2}, c^{2}, d^{2}, abcd\rangle$ which contains a normal free Abelian subgroup generated by $ab$ and $bc$.

Finally in case c) let $H=\langle B,BC^n\rangle$. Then, as above, $H \cong\langle a, b \mid a^{2},(a b^{n})^{2}\rangle$ in which $K=\langle b^{n}\rangle$ is an infinite cyclic normal subgroup such that $H/K=\mathbb{Z}_{2}\ast \mathbb{Z}_{n}, n \geq 2$. Thus in all cases, $G$ is not hyperbolic.
\end{proof}

\begin{cor}Let $G$ be a group of F-type. $G$ is hyperbolic unless $U$ is a proper power or a product of two elements of order 2 and $V$ also is a proper power or a product of two elements of order 2.
\end{cor}

We just want to remark here that by the result of F.~Dahmani and V.~Guirardel, see \cite{58}, the isomorphism problem for a hyperbolic group of F-type is solvable in the class of hyperbolic groups.

It remains the question if the isomorphism problem is solvable for a group of F-type in some class of groups. This is an open problem. There exists just a result in the special case of cyclically pinched one-relator groups, that is, in the case of groups of F-type with $e_1=e_2=\dots=e_n=0$. The isomorphism problem for a cyclically pinched one-relator group $G$ is solvable in the class of one-relator groups, see \cite{193} and \cite{0}.

Certainly, $G$ has no faithful representation in $\PSL(2,\mathbb{R})$ if $U=XY$ with $X^2=Y^2=1$ and $V=WZ$ with $W^2=Z^2=1$ or $U=U_1^k$, $V=V_1^q$ with $k,q\ge 2$.

Concerning the question if there exists a faithful representation $\rho\colon G\to\PSL(2,\mathbb{R})$ we are left, up to symmetry, with the case $U=XY$ with $X^2=Y^2=1$ and $V=V_1^q$ with $q\ge 2$. But in this case there does not exist such a $\rho$ because
\[(\rho(X))^2=(\rho(UY))^2=(\rho(VY))^2=1\] implies $\rho(V_1Y)^2=1$, but $(V_1Y)^2\neq 1$ in $G$ for any essential representation $\rho\colon G\to\PSL(2,\mathbb{R})$.

Our next aim is to characterize the hyperbolicity of a group of F-type by means of faithful representations in $\PSL(2,\mathbb{R})$. In \cite{73} we showed the following.

\begin{thm}Let $G$ be a hyperbolic group of F-type. Then $G$ has a faithful representation in $\PSL(2,\mathbb{R})$.
\end{thm}

Hence, we get the following.

\begin{cor}Let $G$ be a group of F-type. Then $G$ is hyperbolic if and only if $G$ has a faithful representation in $\PSL(2,\mathbb{R})$.
\end{cor}

This has now some additional algebraic consequences. We know that $\PSL(2,\mathbb{R})$ is \textit{commutative transitive}. Hence, any hyperbolic group of F-type is commutative transitive. Recall that a group $H$ is commutative transitive if $[x,y]=1=[y,z]$ for $x,y,z\in H$, $y\neq 1$, then $[x,z]=1$.

Closely tied to commutative transitivity is the concept of being CSA. A group $H$ is CSA or \textit{conjugately separated Abelian} if maximal Abelian subgroups are malnormal. These concepts have played a prominent role in the studies of fully residually free groups, limits groups and discriminating groups see \cite{67}. They also play a role in the solution of the Tarski problems. CSA always implies commutative transitivity.

In general the class of CSA groups is a proven subclass of the class of commutative transitive groups, however they are equivalent in the presence of residual freeness. For hyperbolic groups of F-type we have the following.
\begin{thm}Let $G$ be a hyperbolic group of F-type. Assume that $G$ is torsion-free or has only odd torsion, that is, $e_i$ is odd if $e_i\ge 2$. Then $G$ is CSA.
\end{thm}

\begin{proof}We may assume that $G$ is already a subgroup of $\PSL(2,\mathbb{R})$. Let $A$ be a maximal Abelian subgroup of $G$. Since $G$ is a subgroup of $\PSL(2,\mathbb{R})$ any two elements $a,b\in A$ have, considered as linear fractional transformations, the same fixed points. 

Assume that $G$ is not CSA. Then there are non-trivial elements $a,b\in A$ and $x\in G\setminus A$ with $xax^{-1}=b$. This means that $x$ permutes non-trivially the fixed points of $a$, which are also the fixed points $b$. Hence, $x$ must have order two. This gives a contradiction, and therefore $G$ is CSA.\end{proof}
There is a certain relation between CSA groups and RG groups. A group $H$ is a \textit{restricted Gromov group} or an \textit{RG} group, if $H$ satisfies the following property: If $g$ and $h$ are elements in $H$ then either the subgroup $\langle g,h\rangle$ is cyclic or there exists a positive integer $t$ with $g^t\neq h^t$ and $\langle g^t,h^t\rangle=\langle g^t\rangle\ast\langle h^t\rangle$.

An RG group $H$ has the property that every Abelian subgroup is locally cyclic. Hence, it is often convenient to assume that every element of $H$ is contained in a maximal cyclic subgroup of $H$. This last property is of course satisfied in hyperbolic groups. Certainly, an RG group is commutative transitive. Using  \cite[Theorem 3.9.18]{0} we easily get Theorem \ref{thm4.18a}.

\begin{thm}\label{thm4.18a}Let $G$ be a hyperbolic group of F-type. Assume that $G$ is torsion-free or has only odd torsion. Then $G$ is an RG group.
\end{thm}

We present the following result and sketch its proof.

\begin{thm}Let $G$ be a group of F-type which is not a free product of cyclic groups. Assume that each subgroup of $G$ of infinite index is a free product of cyclic groups. Then $G$ is a co-compact planar discontinuous group.
\end{thm}

\begin{proof}[Sketch of proof]We first consider the case that $G$ is not hyperbolic. We have to consider the situations for $U$ and $V$ which yield that $G$ is not hyperbolic. For these situations, we easily find subgroups of infinite index which are not free products of cyclic groups unless $G\cong\langle U,V\mid U^2=V^2\rangle$ or $G\cong \langle a,b,V\mid a^2=b^2=abV^2=1\rangle$ or $G\cong \langle a,b,c,d\mid a^2=b^2=c^2=d^2=abcd=1\rangle$. But these three groups are co-compact Euclidean planar discontinuous groups.

Now let $G$ be hyperbolic. We just consider $G$ as a quotient of the cyclically pinched one-relator group where the $a_i$ all have infinite order. Next we extend the arguments in \cite{Cio}. Then we use the solution of Nielsen's realization problem by S.\,P\,.~Kerckhoff in \cite{E8} (see also \cite{233} for many details and applications.)\end{proof}

We end this section with the following conjecture.

\begin{conj}Let $H$ be a finitely generated, non-elementary subgroup of $\PSL(2,\mathbb{R})$. Then $H$ is finitely presented.
\end{conj}

\section{One-Relator Amalgamated Products and some Algebraic Consequences}

In this section we give versions and extensions of some of the results in \cite{Rohl}.

If $A,B\in\PSL(2,\mathbb{C})$ we say that the pair $\{A,B\}$ is \textit{irreducible} if $A,B$, regarded as linear fractional transformations, have no common fixed point, that is, $\tr[A,B]\neq 2$.

\begin{thm}[Freiheitssatz]\label{frei}Suppose $G=G_1\Asterisk_A G_2$ with $A=\langle a\rangle$ cyclic. Let $R\in G\setminus A$ be given in a reduced form $R=c_1d_1\cdots c_kd_k$ with $k\ge 1$ and $c_i\in G_i\setminus A$, $d_i\in G_2\setminus A$ for $i=1,\dots,k$. Assume that there exists a representation $\phi\colon G\to \PSL(2,\mathbb{C})$ such that $\phi|_{G_1}$ and $\phi|_{G_2}$ are faithful and the pairs $\{\phi(c_i),\phi(a)\}$ and $\{\phi(d_i),\phi(a)\}$are irreducible for $i=1,\dots,k$.

Then the group $H=G/N(R^m)$, $m\ge 2$, admits a representation $\rho\colon H\to \PSL(2,\mathbb{C})$ such that $G_1\to H\stackrel{\rho}{\to}\PSL(2,\mathbb{C})$ and $G_2\to H\stackrel{\rho}{\to}\PSL(2,\mathbb{C})$ are faithful and $\rho(R)$ has order $m$. In particular, $G_1\to H$ and $G_2\to H$ are injective.
\end{thm}

\begin{proof}Let $\phi\colon G\to \PSL(2,\mathbb{C})$ be the given representation of $G$ such that $\phi|_{G_1}$ and $\phi|_{G_2}$ are faithful and the pairs $\{\phi(c_i),\phi(a)\}$ and $\{\phi(d_i),\phi(a)\}$ are irreducible for $i=1,\dots, k$.

We may assume that $\phi(a)$ has the form $\phi(a)=
\begin{pmatrix}s&0\\0&s^{-1}
\end{pmatrix}$ or 
 $\phi(a)=
\begin{pmatrix}1&1\\0&1
\end{pmatrix}$.

Suppose first that $\phi(a)=
\begin{pmatrix}s&0\\0&s^{-1}
\end{pmatrix}$ and let $T=
\begin{pmatrix}t&0\\0&t^{-1}
\end{pmatrix}$ with $t$ a variable whose value in $\mathbb{C}$ is to be determined.

Define
\begin{itemize}
\item $\rho(h_1)=\phi(h_1)$ for $h_1\in G_1$ and
\item $\rho(h_2)=T\phi(h_2)T^{-1}$ for $h_2\in G_2$.
\end{itemize}
Since $T$ commutes with $\phi(a)$ for any $t$ the map $\rho\colon H\to\PSL(2,\mathbb{C})$ will define a representation with the desired properties if there exists a value $t$ such that $\rho(R)$ has order $m$.

Recall that a complex projective matrix $B$ in $\PSL(2,\mathbb{C})$ will have finite order $m\ge 2$ if $\tr B=\pm 2\cos (\frac{\pi}{m})$. As in the statement of the theorem assume $R=c_1d_1\cdots c_kd_k$ with $k\ge 1$ and $c_i\in G_1\setminus A$, $d_i\in G_2\setminus A$ for $i=1,\dots, k$, and assume that the pairs $\{\phi(c_i),\phi(a)\}$ and $\{\phi(d_i),\phi(a)\}$ are irreducible for $i=1,\dots,k$. Define
\[f(t)=\tr(\phi(c_1)T\phi(d_1)T^{-1}\cdots\phi(c_k)T\phi(d_k)T^{-1}),\] 
then $f(t)$ is a Laurent polynomial in $t$ of degree $2k$ in both $t$ and $t^{-1}$. The coefficients of $t^{2k}$ and $t^{-2k}$ are non-zero because the pairs $\{\phi(c_i),\phi(a)\}$ and $\{\phi(d_i),\phi(a)\}$ are irreducible.

Therefore by the fundamental theorem of algebra there exists a $t_0$ with $f(t_0)=2\cos\frac{\pi}{m}.$ With this choice of $t_0$ we have $\tr(\rho(R))=2\cos\frac{\pi}{m}$ and this $\rho(R)$ has order $m$. Therefore $\rho$ is a representation with the desired properties. Now assume $\phi(a)=\begin{pmatrix}1&1\\0&1\end{pmatrix}$. In this case, define $T=
\begin{pmatrix}1&t\\0&1
\end{pmatrix}$ with $t$ again a variable. Again $T$ commutes with $\rho(a)$ and the proof goes through analogously as above giving the desired representation.
\end{proof}

\begin{cor}\label{cor1}Let $G=G_1\Asterisk_A G_2$ be a group of F-type. Assume further that $n\ge 4$, $2\le p\le n-2$ and neither $U=U(a_1,\dots,a_p)$ nor $V=V(a_{p+1},\dots,a_n)$ is a proper power in the free product on the generators they involve.

Suppose that $UV$ involves all the generators and let $R=c_1d_1\cdots c_kd_k$ with $k\ge 1$ and $c_i\in\langle a_1,\dots,a_p\rangle\setminus\langle U\rangle$, $d_i\in\langle a_{p+1},\dots,a_n\rangle\setminus\langle V\rangle$ for $i=1,\dots,k$, and let $m\ge 2$. Then the conclusion of Theorem \ref{frei} holds for $H=G/N(R^m)$ with 
\[ G_1=\langle a_1,\dots, a_p\mid a_1^{e_1}=\dots=a_p^{e_p}=1\rangle,\]
\[ G_2=\langle a_{p+1},\dots, a_n\mid a_{p+1}^{e_{p+1}}=\dots=a_n^{e_n}=1\rangle,\]
and
\[A=\langle U^{-1}\rangle =\langle V\rangle.\]
\end{cor}

\begin{proof}By Theorem \ref{theo3}, the group $G$ admits a faithful representation $\phi$ in $\PSL(2,\mathbb{C})$ such that the pairs $\{\phi(c_i),\phi(U)\}$ and $\{\phi(d_i),\phi(V)\}$ are irreducible for $i=1,\dots, k$.
\end{proof}

\begin{cor}Let $G$ be a group of F-type as in Corollary \ref{cor1} and $R$ the relator as in Corollary \ref{cor1}. Suppose $m\ge 8$. Then $H=G/N(R^m)$ is virtually torsion-free.
\end{cor}

\begin{proof}In $G$ any element of finite order is conjugate to a power of a generator $a_i$. Since $m\ge 8$ from the torsion theorem for small cancellation products of D.~Collins and F.~Perraud, see \cite{149} or \cite{Collins}, any element of finite order in $G$ is conjugate to a power of a generator or a power of $R$.

Since $\rho(R)$ has exact order $m$ an the representations of $G_1$ and $G_2$ are faithful it follows that (the constructed) $\rho$ is essentially faithful and therefore $G$ is virtually torsion-free.
\end{proof}

\begin{rem}Let $G$ be a group of F-type with the presentation \[G=G_1\Asterisk_A G_2\]
with $G_1\neq A\neq G_2$, \[ G_1=\langle a_1,\dots, a_p\mid a_1^{e_1}=\dots=a_p^{e_p}=1\rangle,\]
\[ G_2=\langle a_{p+1},\dots, a_n\mid a_{p+1}^{e_{p+1}}=\dots=a_n^{e_n}=1\rangle,\] $U=U(a_1,\dots,a_p)$, $V=(a_{p+1},\dots,a_n)$, and $A=\langle U^{-1}\rangle=\langle V\rangle$.

We call $G$ a \textit{special group of F-type} if $UV$ involves all the generators, $n\ge 4$, $2\le p\le n-2$, and neither $U$ nor $V$ is a proper power in the free product on the generators they involve.
\end{rem}

Now let $G$ be a special group of F-type and let $R\in G\setminus A$ be given in a reduced form $R=c_1d_1\cdots c_kd_k$ with $k\ge 1$ and $c_i\in G_1\setminus A$, $d_i\in G_2\setminus A$ for $i=1,\dots,k$.

Then we may apply the theory of one-relator quotients $H=G/N(R^m)$ with $m\ge 2$ as described in \cite{82} and the deficiency arguments in Section 2 analogously in this more general context.

If we mirror the respective proofs there, then we easily get the following.

\begin{thm}Let $G$ be a special group of F-type, $R$ the relator as in Corollary \ref{cor1}, and $H=G/N(R^m)$, $m\ge 2$. Then the following hold.
\begin{enumerate}
\item[1.] For $i=1,\dots, n$ let $\alpha_i=0$ if $e_i=0$ and $\alpha_i=\frac{1}{e_i}$, if $e_i\ge 2$. Then
\begin{enumerate}
\item[(i)] if $\sum_{i=1}^n\alpha_i+\frac{1}{m} \le n- 2$, then $H$ has a subgroup of finite index mapping homomorphically on $\mathbb{Z}$. In particular, $H$ is infinite.
\item[(ii)] if $\sum_{i=1}^n\alpha_i+\frac{1}{m}<n-2$, then $H$ has a subgroup of finite index mapping homomorphically onto a free group of rank 2. In particular, $H$ is SQ-universal.
\end{enumerate}
\item[2.] $H$ is a non-trivial free product with amalgamation.
\item[3.] If $n\ge 5$ or $n=4$ and at least one of the $e_i$ is not equal to 2, then $H$ has a free subgroup of rank 2.
\end{enumerate}
\end{thm}

The details can be found in \cite{Rohl}.

We have considered special groups $G$ of F-type. We can make similar calculations for the hyperbolic groups of F-type which also have a faithful representation on $\PSL(2,\mathbb{C})$.

If now, for instance, $V=V_1^q$, $q\ge 2$, $V_1$ not a proper power in $\langle a_{p+q},\dots,a_n\rangle$, one has to be careful with the relator $R=c_1d_1\cdots c_kd_k$ for $k\ge 1$ where $c_i\in\langle a_1,\dots, a_p\rangle\setminus\langle U\rangle$ and $d_i\in\langle a_{p+1},\dots, a_n\rangle\setminus \langle V\rangle$ for $i=1,\dots,k$; $G$ admits a faithful representation $\phi$ in $\PSL(2,\mathbb{C})$ but then a pair $\{\phi(d_i),\phi (V)\}$ is not irreducible if $d_i=V_1^t$ with $q\nmid t$. But everything goes through analogously if we consider only relators $R$ of the form as above with $d_i\in\langle a_{p+1},\dots,a_n\rangle\setminus\langle V_1\rangle$.


\begin{thebibliography}{[88]}

\bibitem{1}M. Aab and G. Rosenberger, Subgroup separable free products with cyclic amalgamations. Results Math., 28:185–194, 1995.

\bibitem{All}R. B. J. T. Allenby, Conjugacy separability of a class of one-relator groups, Proc. Amer. Math. Soc. 116, 621-628, 1992.

\bibitem{E2}G. Baumslag, J. W. Morgan, and P. Shalen, Generalized triangle groups, Math. Proc. Comb. Phil. Soc. 102, 25-31, 1987.

\bibitem{9}B. Baumslag and S. Pride, Groups with two more generators than relators. J. Lond. Math. Soc.,
17:425–426, 1978.

\bibitem{27}M. R. Bridson, Geodesics and Curvature in Metric Simplicial Complexes. Cornell University, 1991.

\bibitem{29}K. Brown, Cohomology of Groups. Springer, 1982.

\bibitem{Cio}L. Ciobanu, B. Fine and G. Rosenberger, Classes of Groups Generalizing a Theorem of Benjamin Baumslag, Communications in Algebra, 44:2, 656-667, 2016.

\bibitem{Collins}D. Collins and F. Perraud, Cohomology and finite subgroups of small cancellation quotients of free products, Math. Proc. Camb. Phil. Soc. 37(1984), 243-359.

\bibitem{49}D. Collins and H. Zieschang, On the Nielsen method in free products with amalgamated
subgroups. Math. Z., 197:97–118, 1987.

\bibitem{E4}M. Culler and P. Shalen, Varities of Group Representations for Discrete Groups and Splittings of Three Manifolds. Ann. of Math. 117, 109-147, 1983.

\bibitem{58}F. Dahmani and V. Guirardel, The isomorphism problem for all hyperbolic groups. Geom.
Funct. Anal., 21:223–300, 2011.

\bibitem{Rohl}B. Fine, F. Röhl and G. Rosenberger, A Freiheitssatz for certain one-relator amalgamated products, in Combinatorial Geometric Group Theory, pages 73-86, London Math. Soc., Lecture Notes Ser. 204, 1995.

\bibitem{67}B. Fine, A. Gaglione, A. Myasnikov, G. Rosenberger, and D. Spellman, The Elementary Theory
of Groups. De Gruyter, 2014.

\bibitem{73}B. Fine, A. Moldenhauer, and G. Rosenberger, Faithful real presentations of groups of F -type. Int. J. Group Theory, 9:143–155, 2020.

\bibitem{0}B. Fine, A. Moldenhauer, G. Rosenberger and L. Wienke, Topics in Infinite Group Theory, De Gruyter, 2021.

\bibitem{78}B. Fine and G. Rosenberger, Conjugacy separability of Fuchsian groups and related questions. Contemp. Math., 109:11–18, 1990.

\bibitem{81}B. Fine and G. Rosenberger, Classification of All Generating Pairs of Two Generator Fuchsian Groups. In Groups St. Andrews 1993, pages 205–232. London Math. Soc. Lecture Note Ser.
\#211, 1995.

\bibitem{E1}B. Fine and G. Rosenberger, Groups which admit essentially faithful representations. New Zealand J. of Math. 25, 1-7, 1996.

\bibitem{82}B. Fine and G. Rosenberger, Algebraic Generalization of Discrete Groups. Marcel Dekker,
1999.

\bibitem{101}M. Gromov, Hyperbolic groups. In S. M. Gersten, editor, Essays in Group Theory, pages 75–263. Mathematical Sciences Research Institute Publications \#8, 1987.

\bibitem{E5}H. Helling, Diskrete Untergruppen von $SL(2,\mathbb{R})$. Inv. Math. 17, 217-229, 1972.

\bibitem{117}A. Juhász and G. Rosenberger, On the combinatorial curvature of groups of F -type and other one-relator products of cyclics. Contemp. Math., 169:373–384, 1994.

\bibitem{128}A. Karrass and D. Solitar, The subgroups of a free product of two groups with amalgamated subgroup. Trans. Am. Math. Soc., 150:227–255, 1970.

\bibitem{E8}S. P. Kerckhoff, The Nielsen Realization Problem. Ann. of Math. 117, 235-265, 1983.

\bibitem{149}R. C. Lyndon and P. E. Schupp, Combinatorial Group Theory. Springer, 1977.

\bibitem{E3}A. Malcev, On isomorphic matrix representations of infinite groups. Mat. Sb. 8(50), 405-422, 1940.

\bibitem{E7}P. M. Neumann, The SQ-universality of some finitely represented groups. J. Australian Math. Soc. 11, 1-6, 1973.

\bibitem{E6}G. Rosenberger, Some remarks on a paper of A. F. Beardon and P. L. Waterman about strongly discrete subgroups of $SL_2(\mathbb{C})$. J. London Math. Soc. 27, 39-42, 1983.

\bibitem{193}G. Rosenberger, The isomorphism problem for cyclically pinched one-relator groups. J. Pure Appl. Algebra, 95:75–86, 1994.

\bibitem{200}A. Selberg, On Discontinuous Groups in Higher-Dimensional Symmetric Spaces. Int. Colloq. Function Theory. Tata Institute, Bombay, 1960.

\bibitem{Tang}C. Y. Tang, Some results on one-relator quotients of free products, Contemporary Math. AMS 109, 165-177, 1990.

\bibitem{152}W. Magnus, A. Karrass, and D. Solitar. Combinatorial Group Theory. Wiley, New York, 1966.

\bibitem{233}H. Zieschang, Finite Groups of Mapping Classes of Surfaces. Lecture Notes in Math. \#875, Springer, 1981.



\end{thebibliography}
\end{document}